\documentclass[11pt,twoside]{article}
\usepackage{amsmath,amsthm,latexsym,amssymb,graphicx,hyperref}
\usepackage{datetime}
\usepackage{dsfont}
\usepackage{mathrsfs}
\usepackage{color}

\newcommand{\thmx}{Theorem}

\newcommand{\lemx}{Lemma}

\date{}

\begin{document}

\centerline{}

\centerline{\bf }

\centerline{\Large{\bf\boldmath On Properties of Generalized Bi-$\Gamma$-Ideals of}}

\centerline{}

\centerline {\Large{\bf\boldmath $\Gamma$-Semirings\footnote{This research is supported by the Group for Young Algebraists in University of Phayao (GYA), Thailand.}}}

\centerline{}

\centerline{\bf Teerayut Chomchuen and Aiyared Iampan\footnote{Corresponding author. Email: \texttt{aiyared.ia@up.ac.th}}}

\centerline{}

\centerline{Department of Mathematics, School of Science}

\centerline{University of Phayao, Phayao 56000, Thailand}

%
\theoremstyle{plain} 
\newtheorem{theorem}{Theorem}[section]
\newtheorem{lemma}[theorem]{Lemma}
\newtheorem{corollary}[theorem]{Corollary}
\newtheorem{proposition}[theorem]{Proposition}
\newtheorem{claim}[theorem]{Claim}
\theoremstyle{definition} 
\newtheorem{definition}[theorem]{Definition}
\newtheorem{remark}[theorem]{Remark}
\newtheorem{example}[theorem]{Example}
\newtheorem{notation}[theorem]{Notation}
\newtheorem{assertion}[theorem]{Assertion}

\begin{abstract}
The notion of $\Gamma$-semirings was introduced by Murali Krishna Rao \cite{Rao} as a generalization of the notion of $\Gamma$-rings as well as of semirings. We have known that the notion of $\Gamma$-semirings is a generalization of the notion of semirings.
In this paper, extending Kaushik, Moin and Khan's work \cite{Kaushik}, we generalize the notion of generalized bi-$\Gamma$-ideals of $\Gamma$-semirings and investigate some related properties of generalized bi-$\Gamma$-ideals.
\end{abstract}




\section{Introduction and Preliminaries}\numberwithin{equation}{section}
The notion of $\Gamma$-semirings was introduced and studied in 1995 by Murali Krishna Rao \cite{Rao} as a generalization of the notion of $\Gamma$-rings as well as of semiring, and the notion of generalized bi-ideals was first introduced for rings in 1970 by Sz\'{a}sz \cite{Szasz,Szasz2} and then for semigroups by Lajos \cite{Lajos}.
Many types of ideals on the algebraic structures were characterized by several authors such as:
In 2000, Dutta and Sardar \cite{Dutta} studied the characterization of semiprime ideals and irreducible ideals of $\Gamma$-semirings.
In 2004, Sardar and Dasgupta \cite{Sardar} introduced the notions of primitive $\Gamma$-semirings and primitive ideals of $\Gamma$-semirings.
In 2008, Kaushik, Moin and Khan \cite{Kaushik} introduced and studied bi-$\Gamma$-ideals in $\Gamma$-semirings, Pianskool, Sangwirotjanapat and Tipyota \cite{Pianskool} introduced and studied valuation $\Gamma$-semirings and valuation $\Gamma$-ideals of a $\Gamma$-semiring, and Chinram \cite{Chinram2} gave some properties of quasi-ideals in $\Gamma$-semirings.
In 2009, Jagatap and Pawar \cite{Jagatap} introduced the concept of minimal quasi-ideal in $\Gamma$-semirings.
Some properties of minimal quasi-ideals in $\Gamma$-semirings are provided.
In 2010, Ghosh and Samanta \cite{Ghosh} studied the relation between the fuzzy left (respectively, right) ideals of $\Gamma$-semirings and that of operator semiring.
In 2011, Dutta, Sardar and Goswami \cite{Dutta2} introduced different types of operations on fuzzy ideals of $\Gamma$-semirings and proved subsequently that these operations give rise to different structures such as complete lattice, modular lattice on some restricted class of fuzzy ideals of $\Gamma$-semirings.
In 2012, Bekta\c{s}, Bayrak and Ersoy \cite{Bektas} introduced and studied the characterization of soft $\Gamma$-semirings and soft sub-$\Gamma$-semiring.

The concept of ideals for many types of $\Gamma$-semirings is the really interested and important thing in $\Gamma$-semirings.
Therefore, we will introduce and study generalized bi-$\Gamma$-ideals of $\Gamma$-semirings in the same way as
of bi-$\Gamma$-ideals of $\Gamma$-semirings which was studied by Kaushik, Moin and Khan \cite{Kaushik}.


To present the main results we first recall the definition of a $\Gamma$-semiring which is important here and discuss some elementary definitions that we use later.

\begin{definition}\label{M1}\cite{Rao}
Let $M$ and $\Gamma$ be two additive commutative semigroups. Then $M$ is called a \textit{$\Gamma$-semiring} if there exists a mapping $\cdot:M\times\Gamma\times M\rightarrow  M$ (the image $\cdot(a,\alpha,b)$ to be denoted by $a\alpha b$ for all $a,b,c\in M$ and $\alpha,\beta\in \Gamma$) satisfying the following conditions:
\begin{itemize}
\item[(1)] $ a\alpha(b+c) = a\alpha b + a\alpha c$,
\item[(2)] $(a+b)\alpha c = a\alpha c + b\alpha c$,
\item[(3)] $a(\alpha+\beta)b =  a\alpha b + a\beta b$,
\item[(4)] $a\alpha(b\beta c) = a\alpha b(\beta c)$
\end{itemize}
for all $a,b,c\in M$ and $\alpha,\beta\in\Gamma$.
\end{definition}

Let $M$ be a $\Gamma$-semiring, $A$ and $B$ nonempty subsets of $M$, and $\Lambda$ a nonempty subset of $\Gamma$.
Then we define
\begin{center}
$A+B:=\{a+b\mid a\in A$ and $b\in B\}$
\end{center}
and
\begin{center}
$A\Lambda B:=\Biggr\{\displaystyle \sum_{i=1}^{n}a_{i}\lambda_{i} b_{i}\mid n\in\mathds{Z}^{+}, a_{i}\in A, b_{i}\in B$ and $\lambda_{i}\in \Lambda$ for all $i\Biggr\}$.
\end{center}
If $A = \{a\}$, then we also write $\{a\}+B$ as $a+B$, and $\{a\}\Lambda B$ as $a\Lambda B$,
and similarly if $B = \{b\}$ or $\Lambda = \{\lambda\}$.


\begin{example} \cite{Jagatap}
Let $\mathds{Q}$ be set of rational numbers.
Let $(S,+)$ be the commutative semigroup of all $2\times3$ matrices over $\mathds{Q}$ and
$(\Gamma,+)$ commutative semigroup of all $3\times2$ matrices over $\mathds{Q}$.
Define $W\alpha Y$ usual matrix product of $W,\alpha$ and $Y$ for all $W, Y\in S$ and for all $\alpha\in\Gamma$.
Then $S$ is a $\Gamma$-semiring but not a semiring.
\end{example}


\begin{example} \cite{Jagatap}
Let $\mathds{N}$ be the set of natural numbers and $\Gamma= \{1, 2, 3\}$.
Then $(\mathds{N},\max)$ and $(\Gamma,\max)$ are commutative semigroups.
Define the mapping $\mathds{N}\times\Gamma\times\mathds{N}\rightarrow\mathds{N}$,
by $a\alpha b=\min\{a,\alpha,b\}$ for all $a,b\in\mathds{N}$ and $\alpha\in\Gamma$.
Then $\mathds{N}$ is a $\Gamma$-semiring.
\end{example}


\begin{example} \cite{Jagatap}
Let $\mathds{Q}$ be set of rational numbers and
$\Gamma=\mathds{N}$ the set of natural numbers.
Then $(\mathds{Q},+)$ and $(\mathds{N},+)$ are commutative semigroups.
Define the mapping $\mathds{Q}\times\Gamma\times\mathds{Q}\rightarrow\mathds{Q}$,
by $a\alpha b$ usual product of $a,\alpha,b$; for all $a,b\in\mathds{Q}$ and $\alpha\in\Gamma$.
Then $\mathds{Q}$ is a $\Gamma$-semiring.
\end{example}


\begin{example} \cite{Bektas}
For consider the additively abelian groups $\mathds{Z}_{8} = \{0, 1, 2, 3, 4, 5, 6, 7\}$
and $\Gamma = \{2, 4, 6\}$. Let $\cdot:\mathds{Z}_{8}\times\Gamma\times\mathds{Z}_{8}\rightarrow\mathds{Z}_{8}$, $(y,\alpha,s)=y\alpha s$. Then $\mathds{Z}_{8}$ is a $\Gamma$-semiring.
\end{example}


\begin{definition}
A nonempty subset $A$ of a $\Gamma$-semiring $M$ is called
\begin{itemize}
\item [(1)] a \textit{sub-$\Gamma$-semiring} of $M$ if $(A,+)$ is a subsemigroup of $(M,+)$ and $a\gamma b\in A$ for all $a,b\in A$ and $\gamma\in\Gamma$.
\item [(2)] a \textit{$\Gamma$-ideal} of $M$ if $(A,+)$ is a subsemigroup of $(M,+)$, and $x\gamma a\in A$ and $a\gamma x\in A$ for all $a\in A,x\in M$ and $\gamma\in\Gamma$.
\item [(3)] a \textit{quasi-$\Gamma$-ideal} of $M$ if $A$ is a sub-$\Gamma$-semiring of $M$ and $A\Gamma M\cap M\Gamma A\subseteq A$.
\item [(4)] a \textit{bi-$\Gamma$-ideal} of $M$ if $A$ is a sub-$\Gamma$-semiring of $M$ and $A\Gamma M\Gamma A\subseteq A$.
\item [(5)] a \textit{generalized bi-$\Gamma$-ideal} of $M$ if $A\Gamma M\Gamma A\subseteq A$.
\end{itemize}
\end{definition}

\begin{remark}\label{1.8}
Let $M$ be a $\Gamma$-semiring. We have the following:
\begin{itemize}
\item [(1)] Every quasi-$\Gamma$-ideal of $M$ is a bi-$\Gamma$-ideal.
\item [(2)] Every bi-$\Gamma$-ideal of $M$ is a generalized bi-$\Gamma$-ideal.
\end{itemize}
\end{remark}


\begin{definition}\label{1.9}
A $\Gamma$-semiring $M$ is called a \textit{GB-simple $\Gamma$-semiring} if $M$ is the unique generalized bi-$\Gamma$-ideal of $M$.
\end{definition}




\section{Properties of Generalized Bi-$\Gamma$-Ideals}

Before the characterizations of generalized bi-$\Gamma$-ideals of $\Gamma$-semirings for the
main results, we give some auxiliary results which are necessary in what follows.
By \lemx~\ref{1.8} (2) and \cite{Kaushik}, we have the following lemma.

\begin{lemma}\label{M4}
Let $M$ be a $\Gamma$-semiring and $a \in M$. Then $a\Gamma M$ and $M\Gamma a$ are generalized bi-$\Gamma$-ideals of $M$.
\end{lemma}


\begin{lemma}\label{1.5}
Let $M$ be a $\Gamma$-semiring and $\{B_{i}\mid i\in I\}$ a nonempty family of generalized bi-$\Gamma$-ideals of $M$ with $\displaystyle\bigcap_{i\in I} B_{i} \neq \emptyset$.
Then $\displaystyle\bigcap_{i\in I} B_{i}$ is a generalized bi-$\Gamma$-ideal of $M$.
\end{lemma}

\begin{proof}
For all $i\in I$, we have
\begin{center}
 $\displaystyle\left(\bigcap_{i\in I} B_{i}\right)\Gamma M\Gamma\left(\bigcap_{i\in I} B_{i}\right)\subseteq B_{i}\Gamma M\Gamma B_{i}\subseteq B_{i}$.
\end{center}
Thus
\begin{center}
 $\displaystyle\left(\bigcap_{i\in I} B_{i}\right)\Gamma M\Gamma\left(\bigcap_{i\in I} B_{i}\right)\subseteq\bigcap_{i\in I} B_{i}$.
\end{center}
Hence $\displaystyle\bigcap_{i\in I} B_{i}$ is a generalized bi-$\Gamma$-ideal of $M$.
\end{proof}


\begin{lemma}\label{M3}
Let $M$ be a $\Gamma$-semiring and $\emptyset \neq  A \subseteq M$. Then
\begin{equation}\label{Eq5}
A \cup A \Gamma M \Gamma A
\end{equation}
is the smallest generalized bi-$\Gamma$-ideal of $M$ containing $A$.
\end{lemma}

\begin{proof}
Let $B=A\cup A\Gamma M\Gamma A$. Then $A\subseteq B$.
Therefore
\begin{eqnarray*}
B \Gamma M\Gamma B
&=& ( A \cup A \Gamma M \Gamma A) \Gamma M \Gamma ( A \cup A \Gamma M \Gamma A) \\
&\subseteq& [A(\Gamma M \Gamma)(A \cup A\Gamma M \Gamma A)]\cup \\
 && [A \Gamma M \Gamma A(\Gamma M \Gamma)( A \cup A \Gamma M \Gamma A)] \\
 &\subseteq& [A(\Gamma M \Gamma) A\cup A ( \Gamma M \Gamma ) A\Gamma M \Gamma A]\cup \\
 && [A\Gamma M \Gamma A(\Gamma M\Gamma)A\cup A\Gamma M\Gamma A(\Gamma M \Gamma)A\Gamma M\Gamma A] \\
&\subseteq& [A\Gamma M \Gamma A\cup A\Gamma M\Gamma A]\cup[A\Gamma M \Gamma A\cup A\Gamma M\Gamma A]\\
&=& A\Gamma M\Gamma A \\
&\subseteq& A \cup A \Gamma M \Gamma A \\
&=& B.
\end{eqnarray*}
Thus $B=A \cup A\Gamma M\Gamma A$ is a generalized bi-$\Gamma$-ideal of $M$.
We shall show that $B$ is the smallest generalized bi-$\Gamma$-ideal of $M$ containing $A$.
Let $C$ be a generalized bi-$\Gamma$-ideal of $M$ containing $A$.
Then
 \begin{center}
$A\Gamma M\Gamma A\subseteq C\Gamma M\Gamma C\subseteq C$.
\end{center}
Thus
\begin{center}
$B=A\cup A\Gamma M\Gamma A \subseteq C$.
\end{center}
Hence $B$ is the smallest generalized bi-$\Gamma$-ideal of $M$ containing $A$.
\end{proof}


By \lemx~\ref{M3}, let $(A)$ be the smallest generalized bi-$\Gamma$-ideal of $M$ containing $A$.
Therefore
\begin{equation}\label{Eq5.1}
(A)=A\cup A \Gamma M\Gamma A.
\end{equation}
It is also very common to denote the smallest generalized bi-$\Gamma$-ideal of $M$ containing $\{a\}$ as $(a)$.


\begin{lemma}\label{3.18}
Let $T$ be a sub-$\Gamma$-semiring of a $\Gamma$-semiring $M$, $a\in M$ and $(a\Gamma T\Gamma a)\cap T\neq\emptyset$. Then $(a\Gamma T\Gamma a)\cap T$ is a generalized bi-$\Gamma$-ideal of $T$.
\end{lemma}
\begin{proof}
Consider
\begin{eqnarray*}
(a\Gamma T\Gamma a\cap T)\Gamma T\Gamma(a\Gamma T\Gamma a\cap T)
&\subseteq&[(a\Gamma T\Gamma a)\Gamma T\cap T\Gamma T]\Gamma(a\Gamma T\Gamma a\cap T) \\
 &\subseteq&[(a\Gamma T\Gamma a)\Gamma T\cap T]\Gamma(a\Gamma T\Gamma a\cap T) \\
&\subseteq&[[(a\Gamma T\Gamma a\Gamma T)\Gamma(a\Gamma T\Gamma a)]\cap [T\Gamma(a\Gamma T\Gamma a\cap T)]] \\
&\subseteq&[(a\Gamma T\Gamma a)\cap (T\Gamma a\Gamma T\Gamma a)]\cap T \\
&\subseteq&(a\Gamma T\Gamma a)\cap T.
\end{eqnarray*}
Hence $(a\Gamma T\Gamma a)\cap T$ is a generalized bi-$\Gamma$-ideal of $T$.
\end{proof}


\begin{lemma}\label{3.5}
Let $M$ be a $\Gamma$-semiring and $a\in M$. Then $a\Gamma M\Gamma a$ is a generalized bi-$\Gamma$-ideal of $M$.
\end{lemma}
\begin{proof}
Consider
\begin{center}
$(a\Gamma M\Gamma a)\Gamma M\Gamma(a\Gamma M\Gamma a)=a\Gamma (M\Gamma a\Gamma M\Gamma a\Gamma M)\Gamma a\subseteq a\Gamma M\Gamma a$
\end{center}
Hence $a\Gamma M\Gamma a$ is a generalized bi-$\Gamma$-ideal of $M$.
\end{proof}


\begin{proposition}\label{5.2}
Let $M$ be a $\Gamma$-semiring and $T$ a sub-$\Gamma$-semiring of $M$.
Then every subset of $T$ containing $M\Gamma T$ is a sub-$\Gamma$-semiring of $M$.
\end{proposition}

\begin{proof}
Let $A$ be a subset of $T$ such that $M\Gamma T\subseteq A$.
Then
\begin{center}
$A\Gamma A\subseteq M\Gamma T\subseteq A$.
\end{center}
Hence $A$ is a sub-$\Gamma$-semiring of $M$.
\end{proof}


\begin{proposition}\label{6}
Let $M$ be a $\Gamma$-semiring and $T$ a $\Gamma$-ideal of $M$.
Then every subset of $T$ containing $M\Gamma T\cup T\Gamma M$ is a $\Gamma$-ideal of $M$.
\end{proposition}

\begin{proof}
Let $B$ be a subset of $T$ such that $M\Gamma T\cup T\Gamma M\subseteq B$.
Then
\begin{center}
$M\Gamma B\subseteq M\Gamma T\subseteq M\Gamma T\cup T\Gamma M\subseteq B$
\end{center}
and
\begin{center}
$B\Gamma M\subseteq T\Gamma M\subseteq T\Gamma M\cup M\Gamma T\subseteq B$.
\end{center}
Hence $B$ is a $\Gamma$-ideal of $M$.
\end{proof}


\begin{proposition}\label{7}
Let $M$ be a $\Gamma$-semiring and $T$ a quasi-$\Gamma$-ideal of $M$.
Then every subset of $T$ containing $T\Gamma M\cap M\Gamma T$ is a quasi-$\Gamma$-ideal of $M$.
\end{proposition}

\begin{proof}
Let $C$ be a subset of $T$ such that $T\Gamma M\cap M\Gamma T\subseteq C$.
Then
 \begin{center}
$C\Gamma C\subseteq T\Gamma M\cap M\Gamma T\subseteq C$
\end{center}
and
 \begin{center}
$C\Gamma M\cap M\Gamma C\subseteq T\Gamma M\cap M\Gamma T\subseteq C$.
\end{center}
Hence $C$ is a quasi-$\Gamma$-ideal of $M$.
\end{proof}


\begin{proposition}\label{7.1}
Let $M$ be a $\Gamma$-semiring and $T$ a bi-$\Gamma$-ideal of $M$.
Then every subset of $T$ containing $T\Gamma M\Gamma T$ and all of its images is a bi-$\Gamma$-ideal of $M$.
\end{proposition}

\begin{proof}
Let $D$ be a subset of $T$ such that $T\Gamma M\Gamma T\subseteq D$ and $D\Gamma D\subseteq D$.
Then
\begin{center}
$D\Gamma M\Gamma D\subseteq T\Gamma M\Gamma T\subseteq D$.
\end{center}
Hence $D$ is a bi-$\Gamma$-ideal of $M$.
\end{proof}

\begin{proposition}\label{8}
Let $M$ be a $\Gamma$-semiring and $T$ a generalized bi-$\Gamma$-ideal of $M$.
Then every subset of $T$ containing $T\Gamma M\Gamma T$ is a generalized bi-$\Gamma$-ideal of $M$.
\end{proposition}

\begin{proof}
Let $E$ be a subset of $T$ such that $T\Gamma M\Gamma T\subseteq E$.
Then
\begin{center}
$E\Gamma M\Gamma E\subseteq T\Gamma M\Gamma T\subseteq E$.
\end{center}
Hence $E$ is a generalized bi-$\Gamma$-ideal of $M$.
\end{proof}


\begin{theorem}\label{3.6}
Let $M$ be a $\Gamma$-semiring. Then the following statements are equivalent.
\begin{itemize}
\item [$(1)$] $M$ is a GB-simple $\Gamma$-semiring.
\item [$(2)$] $a\Gamma M\Gamma a=M$ for all  $a\in M$.
\item [$(3)$] $(a)=M$ for all  $a\in M$.
\end{itemize}
\end{theorem}

\begin{proof}
$(1)\Rightarrow(2)$ Assume that $M$ is a GB-simple $\Gamma$-semiring and $a\in M$.
By \lemx~\ref{3.5}, we have $a\Gamma M\Gamma a$ is a generalized bi-$\Gamma$-ideal of $M$.
Since $M$ is a GB-simple $\Gamma$-semiring, we have  $a\Gamma M\Gamma a=M$.\\
$(2)\Rightarrow(3)$ Assume that $a\Gamma M\Gamma a=M$ for all $a\in M$ and let $a\in M$.
Then, by \eqref{Eq5.1}, we have
\begin{center}
$(a)=\{a\}\cup a\Gamma M\Gamma a=\{a\}\cup M=M$.
\end{center}
$(3)\Rightarrow(1)$ Assume that $(a)=M$ for all $a\in M$, and let $A$ be a generalized bi-$\Gamma$-ideal of $M$ and $a\in A$.
Then $(a)\subseteq A$.
By assumption, we have
\begin{center}
$M=(a)\subseteq A\subseteq M$.
\end{center}
Thus $M=A$.
Therefore $M$ is a GB-simple $\Gamma$-semiring.
\end{proof}


\begin{lemma}\label{3.8}
Let $B$ be a generalized bi-$\Gamma$-ideal of a $\Gamma$-semiring $M$ and $T$ a sub-$\Gamma$-semiring of $M$.
If $T$ is a GB-simple $\Gamma$-semiring such that $T\cap B\neq \emptyset$, then $T\subseteq B$.
\end{lemma}

\begin{proof}
Assume that $T$ is a GB-simple $\Gamma$-semiring such that $T\cap B\neq \emptyset$ and let $a\in T\cap B$.
By \lemx~\ref{M3}, we have $\{a\}\cup a\Gamma T\Gamma a$ is a generalized bi-$\Gamma$-ideal of $T$.
Since $T$ is a GB-simple $\Gamma$-semiring, we have $\{a\}\cup a\Gamma T\Gamma a=T$.
Thus
\begin{center}
$T=\{a\}\cup a\Gamma T\Gamma a\subseteq B\cup B\Gamma M\Gamma B\subseteq B\cup B\subseteq B$.
\end{center}
Hence $T\subseteq B$.
\end{proof}


\begin{theorem}\label{1.7}
Let $M$ be a $\Gamma$-semiring, $B$ a generalized bi-$\Gamma$-ideal of $M$ and $\emptyset \neq A\subseteq M$.
Then $B\Gamma A$ and $A\Gamma B$ are generalized bi-$\Gamma$-ideals of $M$.
\end{theorem}

\begin{proof}
Since $B$ is a generalized bi-$\Gamma$-ideal of $M$, we have
\begin{center}
$(B\Gamma A)\Gamma M\Gamma(B\Gamma A) = (B\Gamma (A\Gamma M)\Gamma B)\Gamma A\subseteq (B\Gamma M\Gamma B)\Gamma A \subseteq B\Gamma A$
\end{center}
and
\begin{center}
$(A\Gamma B)\Gamma M\Gamma(A\Gamma B) =A\Gamma(B\Gamma (M\Gamma A)\Gamma B)\subseteq A\Gamma (B\Gamma M\Gamma B)\subseteq A\Gamma B$.
\end{center}
Therefore $B\Gamma A$ and $A\Gamma B$ are generalized bi-$\Gamma$-ideals of $M$.
\end{proof}


\begin{theorem}\label{2.0}
Let $M$ be a $\Gamma$-semiring. Then $M=a\Gamma M\Gamma a$ for all $a\in M$  if and only if  $M$ is a GB-simple $\Gamma$-semiring.
\end{theorem}

\begin{proof}
Assume that $a\Gamma M\Gamma a=M$ for all $a\in M$ and let $B$ is a generalized bi-$\Gamma$-ideal of $M$ and $b\in B$.
By assumption, we have
\begin{center}
$M = b\Gamma M\Gamma b \subseteq B\Gamma M\Gamma B \subseteq B\subseteq M$.
\end{center}
Hence $M=B$, so $M$ is a GB-simple $\Gamma$-semiring.

Conversely, assume that $M$ is a GB-simple $\Gamma$-semiring and $a\in M$.
By \lemx~\ref{M4} and \thmx~\ref{1.7}, we have $a\Gamma M\Gamma a$ is a generalized bi-$\Gamma$-ideal of $M$.
Since $M$ is a GB-simple $\Gamma$-semiring, we have $a\Gamma M\Gamma a=M$.
\end{proof}


\begin{theorem}\label{2.1}
Let $M$ be a $\Gamma$-semiring and $B$ a bi-$\Gamma$-ideal of $M$.
Then $B$ is a minimal generalized bi-$\Gamma$-ideal of $M$ if and only if $B$ is a GB-simple $\Gamma$-semiring.
\end{theorem}

\begin{proof}
Assume that $B$ is a minimal generalized bi-$\Gamma$-ideal of $M$. By assumption, $B$ is a $\Gamma$-semiring.
Let $C$ be a generalized bi-$\Gamma$-ideal of $B$. Then
\begin{equation}\label{Eq6}
C\Gamma B\Gamma C \subseteq C\subseteq B.
\end{equation}
Since $B$ is a generalized bi-$\Gamma$-ideal of $M$ and by \thmx~\ref{1.7}, we have $C\Gamma B\Gamma C$  is a generalized bi-$\Gamma$-ideal of $M$.
Since $B$ is a minimal generalized bi-$\Gamma$-ideal of $M$, we get $C\Gamma B\Gamma C=B$.
Thus, by \eqref{Eq6}, we have $B=C$.
Hence $B$ is a GB-simple $\Gamma$-semiring.

Conversely, assume that $B$ is a GB-simple $\Gamma$-semiring.
Let $C$ be a generalized bi-$\Gamma$-ideal of $M$ such that $C\subseteq B$.
Then
\begin{center}
$ C\Gamma B\Gamma C \subseteq C\Gamma M\Gamma C\subseteq C$.
\end{center}
Thus $C$ is a generalized bi-$\Gamma$-ideal of $B$.
Since $B$ is a GB-simple $\Gamma$-semiring, we have $B=C$.
Hence $B$ is a minimal generalized bi-$\Gamma$-ideal of $M$.
\end{proof}

\begin{theorem}\label{3.11}
Let $M$ be a $\Gamma$-semiring having a proper generalized bi-$\Gamma$-ideal.
Then every proper generalized bi-$\Gamma$-ideal of $M$ is minimal if and only if the intersection of any two distinct proper generalized bi-$\Gamma$-ideals is empty.
\end{theorem}

\begin{proof}
Assume that every proper generalized bi-$\Gamma$-ideal of $M$ is minimal and
let $B_{1}$ and $B_{2}$ be two distinct proper generalized bi-$\Gamma$-ideals of $M$.
By assumption, we have $B_{1}$ and $B_{2}$ are minimal.
We shall show that  $B_{1}\cap B_{2}=\emptyset$.
Suppose that $B_{1}\cap B_{2}\neq\emptyset$.
By \lemx~\ref{1.5}, we have $B_{1}\cap B_{2}$ is a proper generalized bi-$\Gamma$-ideal of $M$.
Since $B_{1}\cap B_{2}\subseteq B_{1}$ and $B_{1}\cap B_{2}\subseteq B_{2}$, we get $B_{1}\cap B_{2}=B_{1}$ and $B_{1}\cap B_{2}= B_{2}$.
Thus $B_{1}= B_{2}$ which is a contradiction.
Hence $B_{1}\cap B_{2}=\emptyset$.

Conversely, assume that the intersection of any two distinct proper generalized bi-$\Gamma$-ideals is empty.
Let $B$ be a proper generalized bi-$\Gamma$-ideal of $M$ and $C$ a generalized bi-$\Gamma$-ideals of $M$ such that $C\subseteq B$.
Suppose that $C\neq B$.
Then $C$ is a proper generalized bi-$\Gamma$-ideal of $M$.
Since $C\subset B$ and by assumption, we have $C=C\cap B=\emptyset$ which is a contradiction.
Therefore $C=B$, so $B$ is minimal.
\end{proof}

\section*{Acknowledgment}
The authors wish to express their sincere thanks to the referees for the valuable suggestions which lead to an improvement of this paper.



{\bf Received: \today}

\end{document}